\documentclass[amscd,amssymb,verbatim,12pt]{amsart}
\usepackage{pdfsync}
\newif\ifAMS
\IfFileExists{amssymb.sty}
  {\AMStrue\usepackage{amssymb}}
  {\usepackage{latexsym}}
\theoremstyle{plain}

\newtheorem{Thm}{Theorem}
\newtheorem{Cor}[Thm]{Corollary}
\newtheorem{Lem}[Thm]{Lemma}

\theoremstyle{definition}
\newtheorem{Def}{Definition}

\theoremstyle{remark}

\newtheorem{Qu}{Question}

\newtheorem{Ex}[Thm]{Example}

\newcommand{\interior}{^{ \kern-5pt ^\circ}}
\newcommand {\bd}{\partial}

\newcommand {\iy}{\infty}

\newcommand {\N}{{\mathbb N}}
\newcommand {\R}{{\mathbb R}}
\newcommand {\Z}{{\mathbb Z}}

\newcommand {\E}{{\mathbb E}}
\newcommand {\HH}{{\mathbb H}}
\newcommand {\SSS}{{\mathbb S}}

\newcommand {\nb}{\text{Nbh}}

\newcommand {\diam}{\text{diam}}
\newcommand {\radius}{\text{radius}}

\begin{document}
\title{On cyclic CAT(0)  domains of discontinuity}

\author
{Eric Swenson }

\subjclass{54F15,54F05,20F65,20E08}

\email [Eric Swenson]{eric@math.byu.edu}
\address
[Eric Swenson] {Mathematics Department, Brigham Young University,
Provo UT 84602}

\begin{abstract} Let $X$ be a CAT(0) space, and $G$ a discrete  cyclic group of isometries of $X$.
We investigate the domain of discontinuity for the action of $G$ on the boundary $\bd X$.
\end{abstract}
\maketitle
\section {Introduction}
The idea of a domain of discontinuity was first investigated in the setting of Kleinian groups, discrete groups of isometries of hyperbolic $m$-space, $\HH^m$.  Using the Poincare ball model  of $\HH^m$, where $\HH^m$ is the interior of the open unit $m$-ball, we see that the boundary of $\HH^m$ is the unit sphere $\SSS^{m-1}$, and we see that isometries of $\HH^m$ extent to homeomorphsims of  the closed $m$-ball, $\HH^m \cup \SSS^{m-1}$.

Let $G$ be a discrete group of isometries of $\HH^m$, and $(g_i)$ a sequence of distinct elements of $G$.  Fixing a point $a \in \HH^m$, using compactness and passing to a subsequence we may assume that $g_i(a) \to p \in \SSS^{m-1}$ and $g^{-1}_i(a) \to n \in \SSS^{m-1}$. It can be shown that  $g_i(x) \to p$ uniformly on compact subsets of $\SSS^{m-1} -\{n\}$.  
\begin{Def}
Let $G$ be a group of homeomorphsim of a compact hausdorf space $Z$, we say that $G$ acts on $Z$ as a discrete convergence group if for each sequence of distinct elements of $G$, there is a subsequence $(g_i) \subset G$ and points $n,p \in Z$ such that $g_i(x) \to p$ uniformly on compact subsets of $Z$.  As shown above,  Kleinian groups of dimension $m$ act as a discrete convergences group on $\SSS^{m-1}$.
\end{Def}
The set of all such $p$ is the limit set of the Kleinian group $G$, $$\Lambda (G) = \{ p \in \SSS^{m-1}: \exists (g_i) \subset G\text{ with
}g_i(x) \to p \text{ for some }x \in \HH^m\}$$ and the domain of discontinuity of $G$, 
$\Omega G = \SSS^{m-1} - \Lambda G$. Clearly $\Lambda G$ is $G$-invariant and  closed, so $\Omega G$ is $G$-invariant and  open.  Let $p \in \Lambda G$ and $(g_i)\subset G$ with $g_i(x) \to p$ and $g_i^{-1}(x) \to n \in \Lambda G$ for some (any) $x \in \HH^m$.  Then
for any $a \in \Omega G$, $g_i(a) \to p$.  Thus if $\Omega G \neq \emptyset$ then $p$ is a limit point of $\Omega G$, and so $\Omega G$ is dense in $\SSS^{m-1}$.

We have the following result: If $G$ is a discrete group of isometries of hyperbolic $m$-space, $\HH^m$, then $\bd \HH^m = \SSS^{m-1}$ is a disjoint union of the limit set of $G$, $\Lambda G$ and the domain of discontinuity of $G$, $\Omega(G)$, where $\Omega G$ is open and either dense or empty.  Since $G$ acts as a convergence group on $\SSS^{m-1}$,  the action of $G$ on $\Omega G$ is properly discontinuous (hence the name domain of discontinuity). 

This result holds whenever $G$ acts as a discrete convergence group on a compact Hausdorff space, in particular when $G$ is a discrete group of isometries of a proper $\delta $-hyperbolic space.  

In this paper we consider the case where the negative curvature ($\delta$-hyperbolic) condition is relaxed to a non-positive curvature (CAT(0)) condition.  Considering the action of $\Z^m$ on $\R^m$
we see that often the action on the boundary may be trivial and so the domain of discontinuity must be empty whenever the group is infinite. Thus the boundary will not in general be the union of the limit set and the domain of discontinuity. 

The conjecture seems to be that if a cyclic subgroup acts ``nicely" on a CAT(0) space and the action on the boundary is not virtually trivial, then there should be an open dense subset of the boundary on which the  cyclic subgroup acts properly discontinuously.  

This conjecture is realized when the Tits diameter of the boundary is large, but still open in the other cases.  The case where the fixed points of the cyclic subgroup are spherical suspension points of the boundary will not be addressed in this note.

The author wishes to thank G. Levitt and K. Ruane for helpful conversations.

\section{Definitions and Basic results} 
We refer the reader to \cite{BRI-HAE} or \cite{BAL} for more details of the following.
\begin{Def} For $X$ a geodesic metric space and $\Delta(a,b,c)$ a geodesic triangle in $X$ with vertices $a,b,c \in X$ there is a {\em comparison }triangle $\bar \Delta=\Delta(\bar a,\bar b, \bar c) \subset \E^2$ with 
$d(a,b) =d(\bar a, \bar b)$, $d(a,c) = d(\bar a, \bar c)$ and $d(b,c)=d(\bar b, \bar c)$.  
We define the comparison angle $\bar \angle_a(b,c) =\angle_{bar a}(\bar b,\bar c)$.  

Each point $z \in \Delta(a,b,c)$ as a unique comparison point, $\bar z \in \bar \Delta$ on the edge corresponding to the edge of $z$ and in the corresponding location on that edge.  We say that the triangle $\Delta(a,b,c)$ is CAT(0) if for any $y, z \in \Delta(a,b,c)$ with comparison points $\bar y, \bar z \in \bar \Delta$, $d(y,z) \le d(\bar y,\bar z)$. The space $X$ is said to be CAT(0) if every geodesic triangle in $X$ is CAT(0). 
\end{Def}
If $X$ is CAT(0), notice that for any geodesics $\alpha:[0, r] \to X$ and $\beta:[0,s] \to X$ with $\alpha(0)=\beta(0)=a$, the function $\theta(r,s) =\bar \angle_{r(0)}(\alpha(r),\beta(s))$ is an increasing function of  $r,s$.  Thus $\lim\limits_{r,s \to 0} \theta(r,s)$ exists and we call this limit $\angle_a(\alpha(r),\beta(s))$.  It follows that for any $a,b,c \in X$, a CAT(0) space, $\angle_a(b,c) \le \bar \angle_a(b,c)$.

The
(visual) boundary, $\bd X$, is the set of equivalence classes of
rays, where rays  are equivalent if they fellow travel.  Given a
ray $R$ and a point $x \in X$ there is a ray $S$ emanating from
$x$ with $R \sim S$.  Fixing a base point $\mathbf 0 \in X$ we
defined a Topology on $\bar X= X \cup \bd X$ by taking the basic
open sets of $ x \in X$ to be the open metric balls about $x$. For
$y \in \bd X$, and $R$ a ray from $\mathbf 0$ representing $y$, we
construct  basic open sets $U(R,n,\epsilon)$ where $n,\epsilon>0$.
We say $z \in U(R,n,\epsilon)$ if the unit speed geodesic,
$S:[0,d(\mathbf 0,z)] \to \bar X$, from $\mathbf 0$ to $z$
satisfies $d(R(n),S(n)) <\epsilon$.  These sets form a basis for a regular
topology on $\bar X$ and $\bd X$. For any $x \in X$ and $u,v \in \bd X$ we can define 
$\angle_x(u,v) $ and $\bar \angle_x(u,v)$ by parameterizing the rays $[x,u)$ and $[x,v)$ by
$t\in [0,\infty)$ and taking the limit as $t\to 0$ and $t \to \iy$ respectively. 

 for $u,v \in \bd X$, we define $\angle(u,v) =
\sup\limits_{p \in X} \angle_p(u,v)$.
 It follows from \cite{BRI-HAE} that $\angle(u,v) =
\overline{\angle}_p(u,v)$ for any $p \in X$.
 Notice that isometries of $X$ preserve the angle between points of $\bd X$.
 The angle defines a path metric, $d_T$ on the set $\bd X$, called the Tits
metric, whose topology is finer than the given topology of $\bd
X$.  Also $\angle(a,b)$ and $ d_T(a,b)$ are equal whenever either
of them are less than $\pi$.

  The set $\bd X$ with the Tits metric is called the Tits  boundary of $X$,
denoted  $TX$.  Isometries of $X$ extend to isometries of $TX$

  The identity function $TX \to \bd X$ is continuous, but the identity
function
 $\bd X \to TX$ is only lower semi-continuous.  That is for any sequences
$(u_n), (v_n)  \subset \bd X$ with $u_n \to u$ and $v_n \to v$ in
$\bd X$, then $$\varliminf d_T(u_n,v_n) \ge d_T(u,v)$$
\begin{Lem} \label{L:horo} Let $v \in \bd X$  and $H \subset X$ a horoball centered at $v$,
then $\bar H \cap \bd X \subset \bar B_T(v, \frac  \pi 2)$ the closed Tits ball  about $v$.
\end{Lem}
\begin{proof} Suppose that $w \in \bar H \cap \bd X$ with $\angle(v,w) > \frac \pi 2 $.
Choose a point $y \in X$ with $\angle_y(v,w) > \frac \pi 2 $.

Let $R:[0,\infty)\to X$ be the geodesic ray from $y$ to $v$, and $b_R:X \to \R$ be the Buseman function associated to $R$, $b_R(x) = \lim\limits_{t \to \infty} [d(x,R(t))-t]$.
Note the $b_R(y) = 0$ and $b_R(R(t)) = -t$ for $t>0$.
By \cite[Exercise II 8.23(1)]{BRI-HAE} any horoball based at $v$ is within finite distance of any other horoball based at $v$.  It follows that all horoballs based at $v$ will have the same limit points in the boundary.  Thus we may assume that $H = b_R^{-1}([-1,-\iy))$. 

Since $w \in \bar H$, there exists $\hat w \in H$ such that $\angle_y(v,\hat w) >\frac \pi 2$.

We recall \cite[Exercise II 8.23(4)]{BRI-HAE}. Choose a geodesic line $\bar R:\R \to \E^2$.
For each $x\in X$ and $n \in \N$ consider a  comparison triangle $\bar \Delta(x, R(0), R(n))$ with $\bar R(i)$ the comparison point for $R(i)$ ($i=0,n$) and $x_n\in \E^2$ the comparison point for $x$.  Choose $r_n \in \R$ with $\bar R(r_n) = \pi_{\bar R}(x_n)$ (where $\pi_{\bar R}$ is orthogonal projection to the line $\bar R$).  Then $r_n \to r$ where $b_R(x)=-r $.  Since $b_R(\hat w) \le -1$,  for $n>>0$, $r_n >0$.  It follows that $\angle_{\bar R(0)}(\bar R(n), \hat w_n) < \frac \pi 2$.  However $$\frac \pi 2 < \angle_y(v,\hat w) =\angle_{R(0)}(R(n), \hat w) \le \bar \angle_{R(0)}(R(n), \hat w) = \angle_{\bar R(0)}(\bar R(n), \hat w_n) < \frac \pi 2$$ which is a contradiction.  
\end{proof}

\section{Large Tits radius}
Recall that for any $u \in \bd X$, $B_T(u, \epsilon) = \{v \in \bd X : d_T(u,v) < \epsilon\}$
and $\bar B_T(u, \epsilon) = \{v \in \bd X : d_T(u,v) \le \epsilon\}$, where $d_T$ is the Tits metric.
For $A \subset \bd X$, we define 
$$\radius_T (A)= \inf\{r: A \subset B_T(u, r) \text{ for some } u \in \bd X\}.$$

If $g$ is an infinite order isometry of $X$, and $\langle g\rangle$ is proper, then $g$ is either hyperbolic or parabolic.  When $g$ is hyperbolic it acts by translation on a line (called an axis of $g$) in $X$ with endpoints $g^+$ (in the direction of translation) and $g^-$ (see \cite{BRI-HAE}).

Recall that a hyperbolic isometry $h$ of $X$ is called {\em rank }1, if $h$ has an axis $L$ which doesn't bound a half flat.  
From \cite{BAL}  we see that $d_T(h^+, \alpha) = \iy$ for all $\alpha \neq h^+$, so if $X$ has a rank 1 isometry then $\radius_T (\bd X)=\iy$

Recall a result from \cite{PA-SW}
\begin{Thm}\label{T:pi} \cite{PA-SW}
Let $X$ be a complete  CAT(0) space and $(g_i) $ a sequence of isometries of $X$
with the
property that $g_i(x) \to p \in \bd X$ and $g_i^{-1}(x) \to n \in
\bd X$ for any $x \in X$.
Then for any $\theta \in [0, \pi]$ and  any compact set $K \subset \bd X
-\overline{B}_T(n,\theta)$, $g_n(K) \to \overline{B}_T(p,  \pi
-\theta)$ (in the sense that for any open $U \supset
\overline{B}_T(p, \pi -\theta)$, $g_n(K) \subset U$ for all $n$
sufficiently large).
\end{Thm}
This Theorem is stated in \cite{PA-SW} with stronger hypothesis, but they were not used in the proof. 
We will refer to this result as $\pi$-convergence.
For a given hyperbolic element $h$ with axis $L$, we set $g_i= h^i$, $n = L(-\iy)= h^-$ and $p = L(\iy)=h^+$.

The following theorem is due to Ballmann.  
\begin{Thm}[Ballmann]\label{T:Ball}  If $h$ is a rank 1 isometry of the complete CAT(0) space $X$, then the group generated by $h$, $\langle h\rangle$ act properly discontinuously on the open set $\bd X - \{h^{\pm }\}$ which is dense if it is non-empty.
\end{Thm}
\begin{proof}  Let $K$ be a compact subset of $\bd X - \{h^{\pm }\}$.  The Tits distance from $K$ to $n =h^-$ will be infinite.  Thus by $\pi$-convergence $h^i(K) \to p= h^+$ and $h^{-i}(K) \to n = h^{-}$.  Thus $\{i \in \Z: h^i(K ) \cap K \neq \emptyset\}$ is finite and the action of $\langle h\rangle$ on $\bd X-\{h^{\pm}\}$ is properly discontinuous.  Clearly $\bd X-\{h^{\pm}\}$ is an open subset and for any $a \in \bd X-\{h^{\pm}\}$, $h^i(a) \to h^+$ and $h^{-i} (a) \to h^{-}$ so $\bd X-\{h^{\pm}\}$ is dense in $\bd X$.
\end{proof}

\begin{Def} For $X$ a complete CAT(0) space we define the limit set 
$\Lambda X \subset \bd X$ to be the set $\{p \in \bd X: \exists (g_i)$ isometries of $X$ with $g_i(x) \to p$ and $g_i^{-1}(x) \to n$ for some $n \in \bd X$ and  for all $x \in X \}$
\end{Def}

The following is a slight generalization of a result of Karlsson \cite{KAR}.
\begin{Thm} Let $X$ be a complete CAT(0) space with $\radius_T(\bd X) > 3\pi$ and $|\bd X|>2$. If $h$ is a hyperbolic isometry of $X$ then $\langle h\rangle$ acts properly discontinuously on the open subset $\Omega= \bd X -[\bar B_T(h^+, \frac \pi 2) \cup \bar B_T(h^{-}, \frac \pi 2)]$. If $[\bar B_T(h^+, \frac \pi 2) \cup \bar B_T(h^{-}, \frac \pi 2)]\subset \Lambda X$ then $\Omega $ is dense in $\bd X$.
\end{Thm}
\begin{proof} By Theorem \ref{T:Ball} we may assume $h$ is not rank 1. Thus some axis of $h$ bounds a half-flat, which corresponds to a Tits geodesic of length $\pi$ from $h^+$ to $h^-$ so $d_T(h^+,h^-) =\pi$.

If there is any point of $w \in \bd X$ which is isolated in the Tits metric, ($d_T(w, v)= \iy, \, \forall v $) then by $\pi$-convergence the orbit of $w$ under $\langle h \rangle$ is infinite and each element of the orbit will be isolated as well.  

By the triangle inequality, for any $q\in \bd X$, there exists $u,v \in \bd X$ such that 
$d_T(q,u),d_T(q,v)\ge \pi$ and $ d_T(u,v) >2\pi$. If follows from $\pi$-convergence that for any $w \in \Lambda X$ and any neighborhood $W$ of $w$ in $\bd X$, the Tits diameter $\diam_T(W) > 2\pi$. 

Since $\diam_T \left[ \bar B_T(h^+, \frac \pi 2) \cup \bar B_T(h^-, \frac \pi 2)\right]\le 2\pi$, $ W \not \subset \bar B_T(h^+, \frac \pi 2) \cup \bar B_T(h^-, \frac \pi 2)$ so 
$W\cap \Omega \neq \emptyset$.  Thus $\Omega$ is dense in $\Lambda X$.. Since closed Tits balls are closed in $\bd X$, $\Omega$ is open.

For any compact $K \subset \Omega$, by $\pi$-convergence
$h^i(K) \to \bar B_T(h^+, \frac \pi 2)$ and $h^{-i}(K) \to \bar B_T(h^{-}, \frac \pi 2)$.  Using this, we can show that  $\{i \in \Z: h^i(K ) \cap K \neq \emptyset\}$ is finite and so the action of $\langle h\rangle$ on  $\Omega$ is properly discontinuous.
\end{proof}
\begin{Ex}
Notice that 
$\Omega$ need not be dense in $\bd X$ if we remove the condition that $\bar B_T(h^\iy, \frac \pi 2) \cup  \bar B_T(h^{-\iy}, \frac \pi 2) \subset \Lambda X$.  You start with the half plane $\{(x,y) : y \ge 0\}$ where $h$ is unit translation in the 1st coordinate.  Now you attach a line at the origin and let $g$ act by translation in that line.  You construct the CAT(0) space $X$ by translating this picture by $g$ and $h$.  This gives you an action of $<g,h>$ on the CAT(0) space $X$.  The isometry $g$ is a rank 1 hyperbolic element, so $X$ is rank 1, but $\Omega= \bd X - [\bar B_T(h^\iy, \frac  \pi 2) \cup  \bar B_T(h^{-\iy}, \frac \pi 2)]$ is not dense in the boundary, in fact its closure hits $ [\bar B_T(h^\iy, \frac \pi 2) \cup  \bar B_T(h^{-\iy}, \frac \pi 2)]$ only in $h^{\pm \iy}$.
\end{Ex}
\begin{Thm} Let $X$ be a proper CAT(0) space with $\radius_T (\bd X)>3\pi$, and $h$ a parabolic isometry of $X$. There exists $m\in \bd X$ fixed point of $h$ such that $\langle h \rangle$ acts properly discontinuously on the open dense subset $\Omega = \bd X - \bar B_T(m,\pi)$ of $\bd X$.
If $\bar B_T(m, \pi) \subset \Lambda X$ then $\Omega $ is dense in $\bd X$.
\end{Thm}
\begin{proof}
Clearly $\Omega$ is non-empty and open.  By \cite[II 8.25]{BRI-HAE}, there exists $m \in \bd X$ such that $h$ leaves invariant each horoball centered at $m$ which implies that $h$ fixes $m$.  Suppose that  $\langle h \rangle$  doesn't act properly discontinuously on $\Omega$.  Then there exists $K $, a compact subset of $\Omega$, and a strictly increasing sequence $(i_j)\subset \N$ such that $K \cap h^{i_j} (K)\neq \emptyset$ for all $j \in \N$.  

Fix $ x \in X$. Passing to a subsequence we may assume that $h^{i_j} (x) \to p\in \bd X$ and $h^{-i_j}(x) \to n\in \bd X$.  Since $h$ leaves the horosphere $S$ centered at $m$ passing through $x$ invariant, then $n,p \in \bd S \subset \bar B_T(m, \frac \pi 2)$ by Lemma \ref{L:horo}.  Since $\bd X$ is Hausdorf, there exists  $U\subset \bd X$ open with $\bar B_T(m, \pi) \subset U$ and $U \cap K= \emptyset$.  Since $\bar B_T(n, \frac \pi 2), \bar B_T(p, \frac \pi 2) \subset \bar B_T(m, \pi) \subset U$, by $\pi$-convergence $h^{i_j}(K) \subset U$ for all $j>>0$ which implies $h^{i_j}(K) \cap K = \emptyset$ for all $j>>0$ contradicting the choice of $(i_j)$. Thus $\langle h\rangle $ acts properly on $\Omega$.

By the triangle inequality, for any $q\in \bd X$, there exists $u,v \in \bd X$ such that 
$d_T(q,u),d_T(q,v)\ge \pi$ and $ d_T(u,v) >2\pi$. If follows from $\pi$-convergence that for any $w \in \Lambda X$ and any neighborhood $W$ of $w$ in $\bd X$, the Tits diameter $\diam_T(W) > 2\pi$. 

Since $\diam_T (\bar B_T(m, \pi))\le 2\pi$, $ W \not \subset\bar B_T(m, \pi)$ so 
$W\cap \Omega \neq \emptyset$.  Thus $\Omega$ is dense in $\Lambda X$.
\end{proof}

\section{Small Tits Radius}
\begin{Def}Recall that a metric space is {\em proper} if closed metric balls are compact, and a metric space is {\em cocompact} if the quotient of the space by the isometry group is compact.
\end{Def}
\subsection*{We now assume that $X$ is a proper CAT(0) space.}

\begin{Def}We say that $a,b \in \bd X$ are antipodes if $d_T(a,b) \ge \pi$.    The suspension of the antipodes $a$ and $b$,   $S_a^b$, is the union of $a$ and $b$ together with of all Tits geodesics from $a$ to $b$ of length $\pi$ (if any). Notice that $S_a^b =S_b^a$.
\end{Def}
 Let $a \in \bd X$ and $(g_i)$ a sequence of isometries of a CAT(0) space $X$.  We say that $(g_i)$ {\em pulls from} $n$ if there is some unit speed geodesic ray $R:[0,\iy) \to X$ representing $n$, and a sequence  $(s_i) \subset [0,\iy)$ with $s_i \to \iy$ such that  the sequence $(g_i(R(s_i)))$ is bounded.  
Clearly this is independent of the ray chosen.   Passing to a subsequence, we may assume that $g_i(R(s_i)) \to b$, and that
the sequence of rays $(g_i(R))$  (each reparametrized) converges uniformly on compact subsets to a geodesic line $L$ with $g_i(-i) \to L(-\iy)$, and $L(0)= b$.  Notice that for some (and so any) $x \in X$,  $g_i(x) \to L(-\iy)$ and 
$g_i^{-1}(x) \to n$.

 For any $x \in \bd X$, passing to a subsequence, we may assume that  $g_i(x) \to\hat x \in \bd X$.   Thus for any compact set $C \subset TX$ (where $TX$ is $\bd X$ with the topology of the Tits metric)   we can define $f:C \to \bd X$ by $f(x)= \hat x$.  In fact we can define $f$ whenever $C$ has a countable dense subset as a subset of $TX$.   By \cite{BRI-HAE} $f:C \to TX$ is Lipschitz with constant one.  Notice that $f(n) = L(-\iy)$.

  \begin{Lem}\label{L:f} In the above setting  for any $a \in \bar B_T(n,\pi)$, $f:[n,a]\to TX$ is an isometric embeding of $[n,a]$ into $S_{L(-\iy)}^{L(\iy)}$.
\end{Lem}
\begin{proof}  
Because $f$ is Lipschitz with constant one, if $d_T(n,a) = d_T(f(n),f(a))$ then $f:[n,a] \to TX$ is an isometry.   Since $f(n) = L(-\iy)$, it suffices to show that $d_T(f(n), f(a)) + d_T(f(a), L(\iy)) = \pi$.  

 Let $\theta= d_T(n,a)$, so by $\pi$-convergence $\pi-\theta\ge d_T(f(a), L(\iy))$.
Since $f$ is Lipschitz with constant one, $d_T(f(n),f(a)) \le \theta$.  However 
\begin{align*}\pi= \theta + (\pi-\theta) \ge d_T(f(n),f(a))+ d_T(f(a), L(\iy)) \ge  d_T(L(-\iy), L(\iy)) \ge \pi \end{align*}
\end{proof}
\begin{Ex}
The function $f$ defined above need not be an embedding on $B_T(n,\pi)$.
Consider the half flat $Y =\R \times [0,\iy)$.  For each $n \in \Z$   glue the quarter flat $X_n=[n,\iy) \times [0,\iy)$  by identifying $(a,b) \in Y$ with $(c,d)\in X_n$ if $a=c$, $b=d$, and $\ln (a+1-n) \ge b$.  
Let $Z$ be the resulting space.   $\bd Z$ is  broom attached at $\iy$ in the $\R$ factor of $Y$.  Notice that $\Z$ acts on $Z$  ($1 \in \Z$ sends $X_n$ to $X_{n+1}$). Letting $g_i$ be subtraction by $i$, we have that $f(\bd Z) =\bd Y$.  
\end{Ex}
\begin{Def} Let $g$ be a hyperbolic isometry of the CAT(0) space $X$.  We define
$S_g = S^{g^+}_{g^{-}}$, that is the suspension of the endpoints of $g$.
\end{Def}

\begin{Lem}\label{L:discrete} The group $\langle g\rangle $ acts discretely on the open set $\Omega =\bd X- S_g$.
\end{Lem}
\begin{proof} It suffices to show that for any point $a \in \bd X$, the limit points of the $\{g^i(a): \, i\in \N \}$ are in $S_g$. Let $b$ be a limit point of $\{g^i(a): \, i\in \N \}$.  Choose an increasing sequence $(i_j) \subset \N$ such that $g^{i_j}(a) \to b$. 

  Let $L$ be an axis of $g$, then the sequence $(g^{i_j})$ pulls from $n = L(-\iy)= g^-$, leaving $L$ invariant.  If $d_T(a,n) \ge \pi$ then by $\pi$-convergence $b = L(\iy)=g^+$.  If $d_T(a,n) < \pi$ then by 
  we construct  the function $f:[n,a] \to TX$ as above and so $f(a) = b$.   
  
 By  Lemma \ref{L:f}  $f$ embeds $[n,a]$ into $S_{L(-\iy)}^{L(\iy)}= S_g$.  Thus $b \in S_g$.
\end{proof}

\begin{Def} An bounded metric space $X$ is almost a product($K$) of $Y$ and 
$Z$ if there is a convex subset $W\cong Y \times Z$ and $X \subset \nb(W, K)$.
\end{Def}

\begin{Lem}\label{L:limprod} Let $X$ be a  proper cocompact CAT(0) space  $X$.  If
there is a convex subset $W$ of $X$ where $W = Y\times Z$, and a non-empty open subset $V \subset \bd X$ with $ V \subset \bd W$ then $X$ is almost a product of $\hat Y$ and $\hat Z$, convex subsets of $X$.  Additionally if $Y$ is cocompact  then  $\hat Y \cong Y$.
\end{Lem}
\begin{proof} By \cite{GEA-ONT} $X$ has almost extendable($C$) geodesics for some $C$.  
That is for any $x,y \in X$ there exists $z \in B(y,C)$ such that the geodesic segment $[x,z]$ extends to a ray from $x$.    

We may assume that $Y$ and $Z$ are (convex) subsets of $W$.  Fix a base point $w \in W$.
Let $\alpha \in V$.  Let $R:[0,\iy) \to W$ be a geodesic ray with $R(0) =w$ and $R(\iy) = \alpha$.  There exists $N \in \N$, $\epsilon >0$ such that $U(\alpha, N,\epsilon)\subset V$.

Using cocompactness, we find $(g_i)\subset G$ and $t_i\to \iy$ such that
$g_i(R_{t_i}) \to L$ uniformly on compact subset  where $R_{t_i}$ is the ray $R$ reparametrize to have domain $[-t_i, \iy)$.  Using the fact that $X$ is proper and passing to a subsequence, there are unbounded sequences $(y_i) \subset Y$, $(z_i) \subset Z$ so that $g_i(w_i) \to \hat w \in X$ where $w_i=(y_i,z_i)$.  Using properness and passing to a subsequence we have that $g_i(W) \to \hat W$, a convex subset of $X$.   Let $Y_i$ be the copy of $Y$ in $W$ which contains $z_i$, and $Z_i$ the copy of $Z$ in $W$ which contains $y_i$, so that $Y_i \cap Z_i = \{w_i\}$.  Passing to a subsequence, we may assume  that $g_i(Y_i) \to \hat Y$, a convex subset of $\hat W$ and $g_i(Z_i) \to \hat Z$, a convex subset of $\hat W$.   Now let $\hat a, \hat b \in \hat W$.  Thus there are sequences $(a_i), (b_i) \subset W$ with $g_i(a_i) \to \hat a$ and $g_i(b_i) \to \hat b$. 
For each $i$, let  $a_i^Y= \pi_{Y_i}(a_i)$, the projection of $a_i$ in $Y_i$. Define $b_i^Y$, $a_i^Z$ and $b_i^Z$ similarly.  For each $i$,  $W$ is a product of $Y_i$ and $Z_i$ so $d(a_i,b_i)^2 = d(a_i^Y,b_i^Y)^2+d(a_i^Z,b_i^Z)^2$.  The sequence $(g_i(a_i^Y))$ will clearly converge to the to the closest point  projection $\pi_{\hat Y}(\hat a)$. Similarly $g_i(a_i^Z) \to \pi_{\hat Z}(\hat a)$,
$g_i(b_i^Y) \to \pi_{\hat Y}(\hat b)$, and $g_i(b_i^Z) \to \pi_{\hat Z}(\hat b)$.
It follows that 
$$d\left(\hat a,\hat b\right)^2 = d\left(\pi_{\hat Y}(\hat a), \pi_{\hat Y}(\hat b)\right)^2+
d\left(\pi_{\hat Z}(\hat a),\pi_{\hat Z}(\hat b)\right)^2$$

 so  $\hat W = \hat Y\times \hat Z$.  
 
When $Y$ admits a cocompact action we can arrange it so that for some $y \in Y$, and $h_i$ isometry of $Y$, $g_i(h_i(y))$ is bounded and so there is an isometric embedding $h:Y \to X$ such that $h(Y) = \hat Y$, and similarly for $Z$. 

 Now we claim that $X \subset \nb(\hat W, C)  $.  Suppose not, then let $x \in X$ with $d(x,\hat W)> C$.  
 Notice that $g^{-1}_i(\bar B(x, C)) \to \alpha$.  Thus for $i >>0$ $g^{-1}_i(\bar B(x, C)) \subset \tilde U(\alpha ,N,\epsilon)$.  
 
 By almost extendable geodesics($C$) there is a point $v_i \in g^{-1}_i(\bar B(x, C))$ such that $v_i$ is on a ray $S:[0,\iy) \to X$ from our base point $ w$.  Notice that for $i >>0$,  $d(x, g_i(W)) >C$, so in that case $d(v_i, W) > 0$, so $S([0,\iy))\not \subset W$.
  
  Since the (unit speed geodesic) rays based at $ w$ give unique representatives of $\bd X$ and every element of $\bd  W$ is represented by a ray in $W$ from $w$, this implies that $S(\iy) \not \in \bd W$, but $S(\iy) \in U(\alpha ,N,\epsilon)\subset \bd W$ which is a contradiction.

\end{proof}

\begin{Cor} 
Let $X$ be a proper  cocompact CAT(0) space and $W$  a convex subset  with cocompact stabilizer in the isometry group of $X$.  If for some nonempty open $V \subset \bd X$, $V \subset \bd W$, then $X$ lies in a uniform neighborhood of $W$.
\end{Cor}
\begin{proof} In the proof of Lemma \ref{L:limprod}, we may choose the sequence $(g_i)$ in the stabilizer of $W$, so $\hat W=W$.
\end{proof}
Unfortunately, the previous result raises more questions than it answers 
\begin{Qu} If $X$ and $Y$ are cocompact proper CAT(0) spaces and $U\subset \bd X$, $V\subset \bd Y$ non-empty
with $U \cong V$, is $\bd X \cong \bd Y$? 
(where $\cong$ is a homeomorphism in cone topology and an isometry in the Tits metric)

What if additionally we have a group $G$ acting geometrically on both $X$ and $Y$?
The answer is clearly yes  if $X$ or $Y$ is rank 1 and the isomorphism $U \cong V$ is $G$-equivariant.  
\end{Qu}
See \cite{CRO-KLE} and \cite{MOO} for related questions.

We now generalize part of a result of Lytchak to our setting.  

\begin{Thm} \cite{LYT}  Let $Z$ be a geodesically complete  finite dimensonal CAT(1) space.
Then $Z$ has a unique decomposition $Z = \SSS^n \ast G_1 \ast \dots \ast G_k \ast Y_1 \ast \dots \ast Y_m$
where $G_j$ is a thick irreducible building  and $Y_j$ is an irreducible (via spherical join) non-building.
\end{Thm}   
In our setting, the Tits boundary $T X$ is not geodesically complete.  Nonetheless, we obtain the following result: If $Z$ is a finite dimensional CAT(1) space then there is a unique decomposition 
$Z = S^n \ast Y$ where $Y$ doesn't have a sphere as a spherical join factor.  

\begin{Lem}\label{L:TG} If $Y$ and $Z$ are CAT(0) spaces and $F$ is a flat sector in $Y \times Z$, then
$\pi_Y(F)$ is a flat sector in $A$.
\end{Lem}
\begin{proof}
Since a flat sector is the nested union of flat triangles (and this is sufficiant).
It suffices to show that the projection of a flat triangle is a flat triangle.

Let $abc$ be a triangle in a CAT(0) space.  Consider the following  condition: For $t \in (0,1)$ if $e \in [a,b]$ and $f \in [a,c]$ with $d(a,e) = t d(a,b)$ and $ d(a,f) =t d(a,c)$ then $d(e,f) = t d(b,c)$.  
Notice that if $abc$ is flat then this condition is satisfied by similar triangles.  If on the the other hand this condition is satisfied, then by similar triangles, the Euclidian comparison angle $\bar \angle_a(b,c) =\bar \angle_a(e,f)$. It follows that $\angle_a(b,c) =\bar \angle_a(b,c)$ and so the triangle $abc$ is flat.  

Thus we may assume that $abc$ is a triangle in satisfying this condition.  For $t \in (0,1)$ choose $e$ and $f$ as above.  Let $a_1,b_1,c_1,e_1,f_1$ be the projections of $a,b,c,e,f$ respectively into $Y$ and
$a_2,b_2, c_2,e_2,f_2$ the projections of $a,b,c,e,f$ respectively into $Z$.  
By similar Euclidean triangles  $d(a_i, e_i) =t d(a_i,b_i)$ and $d(a_i,f_i)= td(a_i,c_i)$ for $i =1,2$.

By hypothesis $d(e,f) = t d(b,c)$ so $d(e,f)^2 = t^2 d(b,c) ^2$.  Thus
$$d(e_1,f_1)^2 +d(e_2,f_2)^2 = t^2d(b_1,c_1) ^2 + t^2d(b_2,c_2)^2$$

By the CAT(0) inequality applied to the triangles $a_ib_ic_i$, $i=1,2$.  
$d(e_i,f_i) \le t d(b_i,c_i)$ for $i=1,2$.  It follows that $d(e_i,f_i) = t d(b_i,c_i)$ for $i=1,2$.
and so the triangle $a_ib_ic_i$ are flat for $i=1,2$.
\end{proof}
\begin{Def}
\cite[I 5.11]{BRI-HAE} For $Y$ and $Z$ metric spaces with metrics bounded by $\pi$, we define
$Y\ast Z$  to be the quotient of the space $Y\times Z \times [0, \frac \pi 2]$ by the identifications
$(y,z,0)= (y,\hat z,0)$ and $(y,z,\frac \pi 2) = (\hat y,z, \frac \pi 2)$.   We denote the class
$(y,z,\theta)$ by $y \cos \theta + z \sin \theta$ so $(y,z,0) = y$ and $(y,z,\frac \pi 2) = z$, and so we have
$Y,Z \subset Y\ast Z$. For $u=y\cos \theta + z \sin \theta, u'=y'\cos \theta' + z' \sin \theta' \in Y\ast Z$ we define $d(u,u')$ by 
$$\cos(d(u,u')) = \cos \theta \cos \theta ' \cos (d(y,y')) +\sin \theta \sin \theta' \cos (d(z,z'))$$
Clearly for $\theta \neq 0 , \frac \pi 2$,  $d(u,u') = \pi$ if and only if $d(y, y') =\pi = d(z,z')$ and $\theta + \theta ' = \frac \pi 2$.

\end{Def} 
\begin{Lem}Let $\alpha$ be a geodesic in $A\ast B$ where $A$ and $B$ are CAT(1) spaces. Notice that
$A$ and $B$ are $\pi$-convex subsets of $A\ast B$.  
If  the image of $\alpha$ misses $B$, then the projection of $\alpha$ to $A$ is a geodesic.
\end{Lem}
\begin{proof}
Let $Y = C_0(A)$ and $Z= C_0(B)$ be the Euclidian cones on $A$ and $B$ respectively.  By \cite{BRI-HAE} $C_0(A\ast B) = X \cong Y \times Z$ and so $A\ast B = \bd (Y\times Z)$. 
By \cite[I 5.15]{BRI-HAE} It suffices to show that the projection of $\alpha $ to $A$ is a local Tits geodesic.
 There is a flat sector $F \subset X$ based at the cone point 
$0$ of $X$ with $\bd F = \alpha$.  We can think of $Y$ and $Z$ as being convex subsets of $X$.  By Lemma \ref{L:TG}, the projection $F_Y$ of $F$ into $Y$ is a Euclidean sector based at the cone point of $Y$ (which is $0$ of course).  By the proof of \cite[I 5.15]{BRI-HAE} the $\bd F_Y$ will be the projection of $\alpha$ to $A$.  It follows that this projection is a Tits geodesic in $A$.  
\end{proof}

 \begin{Thm} \label{T:decomp}If $Z$ is a finite dimensional complete CAT(1) space, then there is a unique decomposition 
$Z = S \ast D$ where $S\cong \SSS^n$ and $D$ doesn't have a non-empty sphere as a spherical join factor.  
\end{Thm}
\begin{proof} Suppose that $Z = \SSS^k\ast U$ and $Z = \{a,b\} \ast A$, where $\{a,b\} \not \subset \SSS^k$.  It follows that $a$ is the unique antipode to $b$ in $Z$ and visa versa (so $\{a,b\}\cap \SSS^k= \emptyset$).  By the metric on $Z = \SSS^k \ast U$, we have $a_1,b_1 \in \SSS^k$, 
$a_2,b_2 \in U$ and $\theta \in  (0,\pi/2]$ with $a = a_1\cos \theta +a_2\sin \theta$ and $ b = b_1\sin \theta + b_2\cos \theta$.  If follows that $a_2$ is the unique antipode of $b_2$ in $U$ and visa versa (and similarly for $a_1$ and $b_1$ in $\SSS^k$).  

Since $Z =  \{a,b\} \ast A$ then $Z = S_a^b$, and every point of $Z$ is  on a Tits geodesic from $a$ to $b$.  Fix $u \in U \subset \SSS^k \ast U$;  the point $u$ is on a Tits geodesic from $a$ to $b$  which misses $\SSS^k$.  By Lemma \ref{L:TG}, the projection of this Tits geodesic into $U$  is a Tits geodesic from $a_2$ to $b_2$ passing through $u$.   It follows that $U$ is the union of all Tits geodesics in $U$ from $a_2$ to $b_2$.  This implies by \cite[4.1]{LYT} that $U = \{a_2,b_2\}\ast W$ where $W$ is the set of all points of $U$ at Tits distance $\pi/2$ from both $a_2$ and $b_2$.  Thus by \cite[I 5.15]{BRI-HAE}, $Z= \SSS^k\ast(\{a_2,b_2\} \ast W) = (\SSS^k \ast \{a_2,b_2\}) \ast W= \SSS^{k+1} \ast W$.  Since $Z$ is finite dimensional this process must terminate, and so for some $n$, $Z = \SSS^n \ast  Y$ where $Y$ doesn't have a non-empty sphere as a spherical join factor.  Notice by our construction that $$S= \{
a \in Z:\, Z \text{ is a spherical suspension with suspension point }a \}$$  Thus this decomposition is canonical.
\end{proof}
\begin{Cor}\label{C:sus}  Under the hypothesis of Lemma \ref{L:limprod},  if $Y \cong \E^n$ for some $n$ then $\bd Y \subset S$, the set of suspension points of $\bd X$. \end{Cor}
\begin{proof} By \cite{SWE} $\bd X$ is finite dimension, so $TX$ is finite dimensional,  and $\bd X = S \ast D$ as in Theorem \ref{T:decomp}.
  Since $\E^n $ admits a cocompact isometric group action, $\hat Y \cong Y \cong \E^n$.
$\bd X = \bd (\hat Y \times\hat Z) = ( \bd Y)\ast (\bd \hat Z)$.  Since $\bd \hat Y \cong \bd \E^n= S^{n-1}$ then $\bd \hat Y \subset S$.  Let $\alpha \in \bd Y - S$, so $d_T(\alpha, D) < \pi/2$.
We may assume  $g_i(\alpha) \to \beta$ preforce with $\beta \in \bd \hat Y \subset S$.  However $g_i(D) = D$ for all $i$, so $d_T(\beta, D) \le d_T(\alpha, D) < \pi /2$, but that is a contradiction since every point of $S$ is distance $\pi /2 $ form every point of $D$.  Thus $\bd Y \subset S$.
\end{proof}
\begin{Thm} Let $g$ be a hyperbolic isometry of the cocompact CAT(0) space $X$.  If $\Omega = \bd X - S_g \neq \emptyset$, then $\Omega$ is a dense open subset of $\bd X$, and 
$\langle g \rangle $ acts discretely on $\Omega$. \end{Thm}
\begin{proof} By Lemma \ref{L:discrete}, $\langle g \rangle $ acts discretely on the open set  $\Omega$.  We must show that $\Omega$ is dense whenever it is non-empty.

 Suppose that $\Omega $ is not dense. Then there is an open subset $V$ of $\bd X$ with $V \subset \bd S_g$.  Recall that $S_g \cong \R \times Z$ for some $Z$, where the endpoints of $\R$ are $g^\pm$. Thus by Corollary \ref{C:sus}, $g^\pm \in S$, the set of suspension points of $\bd X$.  It follows that $S_g = \bd X$ as required.
\end{proof}

\end{document}
\bye